\documentclass[11pt]{amsart}
\usepackage{graphics,graphicx,color}
\usepackage{srcltx, bm}
\usepackage{tikz}

\usepackage{amsmath,amssymb,mathrsfs,color}
\usepackage{enumerate}
\usepackage{microtype}

\makeatletter
\@namedef{subjclassname@2020}{\textup{2020} Mathematics Subject Classification}
\makeatother

\newtheorem{lemma}{Lemma}[section]
\newtheorem{theorem}[lemma]{Theorem}
\newtheorem{remark}[lemma]{Remark}
\newtheorem{prop}[lemma]{Proposition}
\newtheorem{coro}[lemma]{Corollary}

\newtheorem{example}[lemma]{Example}

\allowdisplaybreaks
\numberwithin{equation}{section}

\begin{document}

\title[On the Averaging Theorems for Stochastic Perturbation]
{On the Averaging Theorems for Stochastic Perturbation of Conservative Linear Systems}

\author{Jing Guo}
\address{J. Guo: School of Mathematics Sciences,
Dalian University of Technology, 116024 Dalian, P. R. China; Universit\'{e} Paris Cit\'{e} and Sorbonne Universit\'{e}, CNRS, IMJ-PRJ, F-75013 Paris, France}
\email{jingguo062@hotmail.com}

\author{Sergei Kuksin}
\address{S. Kuksin: Universit\'{e} Paris Cit\'{e} and Sorbonne Universit\'{e}, CNRS, IMJ-PRG, F-75013 Paris, France; Steklov Mathematical Institute of RAS, 119991 Moscow, Russia; National Research University Higher School of Economics, 119048 Moscow, Russia}
\email{sergei.kuksin@imj-prg.fr}

\author{Zhenxin Liu}
\address{Z. Liu: School of Mathematical Sciences,
Dalian University of Technology, 116024 Dalian, P. R. China}
\email{zxliu@dlut.edu.cn}

\date{}


\begin{abstract}
For stochastic perturbations of linear systems with non-zero pure imaginary spectrum we discuss the averaging theorems in terms of the slow-fast action-angle variables and in the sense of Krylov-Bogoliubov. Then we show that if the diffusion matrix of the perturbation is uniformly elliptic, then in all cases the averaged dynamics does not depend on a hamiltonian part of the perturbation.
\end{abstract}

\maketitle

\section{Introduction}
\setcounter{equation}{0}

We consider a linear system
\begin{equation*}
{\rm{d}} v(t)+Av(t){\rm{d}}t=0, \ v(t)\in\mathbb{R}^{2n},
\end{equation*}
where the operator $A$ does not have Jordan cells and has non-zero pure imaginary eigenvalues. Introducing suitable complex coordinates we write $\mathbb{R}^{2n}$ as a complex space $\mathbb{C}^{n}$, where the operator $A$ takes the diagonal form ${\rm{diag}}\{i\lambda_{j}\}$, and accordingly the system reads
\begin{equation}\label{x6}
{\rm{d}} v_{k}+i\lambda_{k}v_{k} {\rm{d}} t=0, \quad k=1,2,\ldots,n, \quad v_{k}\in \mathbb{C}.
\end{equation}
Below we always use the complex coordinates as in (\ref{x6}), and denote by $\Lambda$ the frequency vector of the system,
\begin{equation*}
\Lambda=(\lambda_{1},\ldots,\lambda_{n})\in(\mathbb{R}\backslash \{0\})^{n}.
\end{equation*}
In this work we analyse the stochastic perturbations of system (\ref{x6})
\begin{equation}\label{z1}
{\rm{d}} v(t)+Av(t){\rm{d}}t=\varepsilon P(v(t)){\rm{d}}t+\sqrt{\varepsilon}\Psi(v(t)){\rm{d}}\beta(t)+\sqrt{\varepsilon}\Theta(v(t)){\rm{d}}\overline{\beta}(t), \ v(0)=v_{0}\in\mathbb{C}^{n}.
\end{equation}
Here $\varepsilon\in(0,1]$, $P(v)$ is a vector field in $\mathbb{C}^{n}$, $\Psi(v)$ and $\Theta(v)$ are complex $n\times n_{1}$-matrices, $\beta(t)$ is the standard complex Wiener process in $\mathbb{C}^{n_{1}}$ and $\overline{\beta}(t)$ is the conjugated process. That is,
\begin{align*}
\beta(t)=(\beta_{1}(t),\ldots,\beta_{n_{1}}(t)),\quad \overline{\beta}(t)=(\overline{\beta_{1}}(t),\ldots,\overline{\beta_{n_{1}}}(t)),
\end{align*}
where $\beta_{j}(t)=\beta_{j}^{R}(t)+i\beta_{j}^{I}(t)$, $\overline{\beta_{j}}(t)=\beta_{j}^{R}(t)-i\beta_{j}^{I}(t)$ and $\{\beta_{j}^{R}(t), \beta_{j}^{I}(t), 1\leq j\leq n_{1}\}$ are standard independent real Wiener processes.
To simplify presentation we restrict ourselves to equations with $\Theta=0$, but suitable versions of all results below hold for equations (\ref{z1}) with non-zero matrices $\Theta(v)$. Passing to the slow time $\tau=\varepsilon t$ we rewrite equation (\ref{z1}) with $\Theta=0$ as
\begin{equation}\label{w1}
{\rm{d}} v_{k}(\tau)+i\varepsilon^{-1}\lambda_{k}v_{k}(\tau){\rm{d}}\tau=P_{k}(v(\tau)){\rm{d}}\tau+\sum\limits_{l=1}^{n_{1}}\Psi_{kl}(v(\tau)){\rm{d}}\beta_{l}(\tau), \ k=1,2,\ldots,n.
\end{equation}
Finally, we pass to the interaction representation
\begin{equation*}
a_{k}(\tau)={\rm{e}}^{i\tau\varepsilon^{-1}\lambda_{k}}v_{k}(\tau),\quad k=1,\ldots,n,
\end{equation*}
so $|a_{k}(\tau)|\equiv|v_{k}(\tau)|$ for all $k$, and write equation (\ref{w1}) in the $a_{k}$\text{-}variables as
\begin{equation}\label{z3}
{\rm{d}} a_{k}(\tau)={\rm{e}}^{i\tau\varepsilon^{-1}\lambda_{k}}P_{k}(v(\tau)){\rm{d}}\tau+{\rm{e}}^{i\tau\varepsilon^{-1}\lambda_{k}}\sum\limits_{l=1}^{n_{1}}\Psi_{kl}(v(\tau)){\rm{d}}\beta_{l}(\tau), \ k=1,2,\ldots,n.
\end{equation}

This paper develops the work \cite{HK}, and is mainly focused on the non-resonant case\footnote{In difference with \cite{HK}, where systems (\ref{w1}) with general frequency vectors $\Lambda$ are examined.}, when for any nonzero integer vector $m=(m_{1},\ldots,m_{n})\in \mathbb{Z}^{n}$ we have
\begin{equation}\label{w2}
\Sigma m_{j}\lambda_{j}\neq0.
\end{equation}
We are concerned with the limiting behavior of actions $I_{k}(\tau)=\frac{1}{2}|v_{k}(\tau)|^{2}=\frac{1}{2}|a_{k}(\tau)|^{2}$ as $\varepsilon\rightarrow0$. In Section 2.1 we discuss it in terms of solutions for equations (\ref{z3}), and then in Section 2.2 in terms of the action-angles coordinates $(I_{k},\varphi_{k}={\rm{Arg}}v_{k}$) for equations (\ref{x6}). Next in Section 3 we prove that in all cases the limiting behavior is independent of a hamiltonian part of the drift $P(v)$. As we show in our second paper \cite{JKL2} this result is instrumental to study averaging in stochastic PDEs.

\subsection*{Notation.}
If $m\geq0$, $E$ is a Banach space and $L$ is $\mathbb{R}^{n}$ or $\mathbb{C}^{n}$, we denote by ${\rm{Lip}}_{m}(L,E)$ the set of locally Lipschitz maps $F:L\rightarrow E$ such that for any $R\geq1$,
\[
(1+|R|)^{-m}\left({\rm{Lip}}(F|_{\overline{B}_{R}(L)})+\sup\limits_{v\in\overline{B}_{R}(L)}|F(v)|_{E}\right)=:\mathcal{C}^{m}(F)<\infty,
\]
where ${\rm{Lip}}(f)$ is the Lipschitz constant of a map $f$ and $\overline{B}_{R}(L)$ is the closed $R$\text{-}ball $\left\{v\big| |v|_{L}\leq R\right\}$. We denote $\mathbb{R}^{n}_{+}=\{x=(x_{1},\ldots,x_{n})\in\mathbb{R}^{n}: x_{j}\geq0 \ \forall j\}$. For a complex matrix $A=(A_{ij})$, $A^{*}$ stands for its Hermitian conjugated matrix: $A^{*}_{ij}=\overline{A_{ji}}$. We denote by $\mathcal{D}(\xi)$ the law of a random variable $\xi$, by $\rightharpoonup$ the weak convergence of measures and by $\mathcal{P}(M)$ the set of Borel measures on the metric space $M$. For complex numbers $z_{1}$, $z_{2}$, we denote their real scalar product by $z_{1}\cdot z_{2}=\mathfrak{R}z_{1}\overline{z_{2}}$. For real numbers $a$ and $b$, $a\vee b$ and $a\wedge b$ indicate their maximum and minimum. For a set $Q$, $1_{Q}$ is its indicator function.

\section{Averaging and effective equation}

We recall (\ref{w2}) and assume that

\textbf{Assumptions }
(A1): The drift $P(v)=(P_{1}(v),\ldots,P_{n}(v))$ belongs to ${\rm{Lip}}_{m_{0}}(\mathbb{C}^{n},\mathbb{C}^{n})$ for some $m_{0}\geq0$. The matrix function $\Psi(v)=(\Psi_{kl}(v))$ belongs to ${\rm{Lip}}_{m_{0}}(\mathbb{C}^{n},{\rm {Mat}}(n\times n_{1}))$.

(A2): The matric function $\Psi(v)$ satisfies one of the following three conditions:

\quad (i) it is $v$-independent;

 or

\quad (ii) it satisfies the non-degeneracy condition: $\Psi(v)\Psi^{*}(v)\geq \alpha E$ \ $\forall v$, for some $\alpha>0$;

or

\quad (iii) it is a $C^{2}$-smooth matrix-function of $v$.

(A3): For any $v_{0}\in\mathbb{C}^{n}$ equation (\ref{w1}) has a unique strong solution $v^{\varepsilon}(\tau;v_{0})$, $\tau\in[0,T]$, which is equal to $v_{0}$ at $\tau=0$. Moreover, there exists $m_{0}'>(m_{0}\vee4)$ such that
\begin{equation}\label{w8}
\mathbf{E}\sup\limits_{0\leq\tau\leq T}|v^{\varepsilon}(\tau;v_{0})|^{2m_{0}'}\leq C_{m_{0}'}(|v_{0}|,T)<\infty.
\end{equation}

Instead of (A3) we may assume a stronger assumption:

(A3$^\prime$): For any $v_{0}\in\mathbb{C}^{n}$ equation (\ref{w1}) has a unique strong solution $v^{\varepsilon}(\tau;v_{0})$, $\tau\geq0$, which is equal to $v_{0}$ at $\tau=0$. There exists $m_{0}'>(m_{0}\vee1)$ such that for any $T'\geq0$
\[
\mathbf{E}\sup\limits_{T'\leq\tau\leq T'+1}|v^{\varepsilon}(\tau;v_{0})|^{2m_{0}'}\leq C_{m_{0}'}(|v_{0}|).
\]

We define the (non-resonant) averaging of a vector field $\widetilde{P}\in{\rm{Lip}}_{m_{0}}(\mathbb{C}^{n},\mathbb{C}^{n})$ as
\begin{equation}\label{z7}
\langle\langle \widetilde{P}\rangle\rangle(a)=\frac{1}{(2\pi)^{n}}\int_{\mathbb{T}^{n}}(\Phi_{\omega})\circ \widetilde{P}(\Phi_{-\omega}a){\rm{d}}\omega, \quad \mathbb{T}^{n}=\mathbb{R}^{n}/(2\pi\mathbb{Z}^{n}),
\end{equation}
(see Section 3.1.2 of \cite{HK}). Here for a real vector $\omega=(\omega_{1},\ldots,\omega_{n})$, $\Phi_{\omega}$ is the rotation operator
\begin{equation*}
\Phi_{\omega}:\mathbb{C}^{n}\rightarrow\mathbb{C}^{n}, \quad \Phi_{\omega}={\rm{diag}}\{{\rm{e}}^{i\omega_{1}},\ldots,{\rm{e}}^{i\omega_{n}}\}.
\end{equation*}
For a locally Lipschitz function $f\in{\rm{Lip}}_{m_{0}}(\mathbb{C}^{n},\mathbb{C})$, we define its (non-resonant) averaging $\langle f\rangle$ as
\begin{equation}\label{e2}
\langle f\rangle(a)=\frac{1}{(2\pi)^{n}}\int_{\mathbb{T}^{n}}f(\Phi_{-\omega}a){\rm{d}}\omega
\end{equation}
(note that $\langle f\rangle(a)$ depends only on $(|a_{1}|,\ldots,|a_{n}|)$). Next, we construct the averaged dispersion matrix $B(a)$ for system (\ref{z3}) (cf. Section 7 in \cite{HKP}). To do that, firstly we define the averaging of the diffusion matrix $\Psi\Psi^{*}$ for (\ref{w1}) as
\begin{align}\label{z9}
A(a):=\frac{1}{(2\pi)^{n}}\int_{\mathbb{T}^{n}}({\Phi_{\omega}}\circ\Psi(\Phi_{-\omega}a))
(\Phi_{\omega}\circ\Psi(\Phi_{-\omega}a))^{*}{\rm{d}} \omega.
\end{align}
Then the averaged dispersion matrix $B(a)$ is defined as the principle square root of $A(a)$. That is, $B(a)$ is a Hermitian matrix such that $B(a)^{2}=A(a)$, and $B\geq0$.
\begin{example}\label{e4}
If the dispersion matrix in equation {\rm{(\ref{w1})}} is constant, then
\begin{align*}
A_{kl}=\sum\limits_{j}\frac{1}{(2\pi)^{n}}\int_{\mathbb{T}^{n}}{{\rm {e}}}^{i(\omega_{k}-\omega_{l})}\Psi_{kj}\overline{\Psi_{lj}}{\rm{d}} \omega=\frac{1}{(2\pi)^{n}}\int_{\mathbb{T}^{n}}{{\rm {e}}}^{i(\omega_{k}-\omega_{l})}{\rm{d}} \omega\sum\limits_{j}\Psi_{kj}\overline{\Psi_{lj}}=
\delta_{k-l}\sum\limits_{j}\Psi_{kj}\overline{\Psi_{lj}}.
\end{align*}
So then $A={\rm{diag}}\{b_{1}^{2},\ldots,b_{n}^{2}\}$, $b_{k}^{2}=\sum_{j=1}^{n_{1}}|\Psi_{kj}|^{2}$, and $B={\rm{diag}}\{b_{1},\ldots,b_{n}\}$.
\end{example}
\subsection{Averaged dynamics of actions via the effective equation}
Following \cite{HK}, we define the effective equation for (\ref{z3}) as
\begin{equation}\label{w3}
{\rm{d}} a_{k}(\tau)=\langle\langle P\rangle\rangle_{k}(a(\tau)){\rm{d}} \tau+\sum\limits_{l}B_{kl}(a(\tau)){\rm{d}}\beta_{l}(\tau), \quad k=1,\ldots,n.
\end{equation}
Denote by $I_{k}(a)=\frac{1}{2}|a_{k}|^{2}=\frac{1}{2}|v_{k}|^{2}$ the $k$-th action, and set $I(a)=(I_{1}(a),\ldots,I_{n}(a))$. Usually the actions of solutions for equations (\ref{w1}) are the most important. We now present the averaging theorem for the long time behavior of the actions, proved in \cite{HK} (Sections 4, 8) for any frequency vector $\Lambda$:
\begin{theorem}\label{w4}
Under Assumptions {\rm{A1, A2, A3}}, for any $v_{0}\in\mathbb{C}^{n}$

(i) the effective equation {\rm{(\ref{w3})}} with initial data $v_{0}$ has a unique strong solution $a(\cdot;v_{0})$;

(ii) $\mathcal{D}(a^{\varepsilon}(\cdot;v_{0}))\rightharpoonup \mathcal{D}(a(\cdot;v_{0}))$ \quad in \ $\mathcal{P}(C([0,T];\mathbb{C}^{n}))$ \quad as \ $\varepsilon\rightarrow0$,\\
so

(iii) $\mathcal{D}(I(v^{\varepsilon}(\cdot;v_{0})))\rightharpoonup \mathcal{D}(I(a(\cdot;v_{0})))$ \quad in \ $\mathcal{P}(C([0,T];\mathbb{R}_{+}^{n}))$ \quad as \ $\varepsilon\rightarrow0$,\\
where $a^{\varepsilon}(\tau;v_{0})$ and $v^{\varepsilon}(\cdot;v_{0})$ satisfy equations {\rm{(\ref{z3})}} and {\rm{(\ref{w1})}}, respectively, with the same initial data $v_{0}$, and $I(v^{\varepsilon}(\cdot;v_{0})):=(I_{k}(v^{\varepsilon}))_{1\leq k\leq n}$.
\end{theorem}
This result is a stochastic version of the Krylov-Bogoliubov averaging (e.g. see \cite{BM,JKW}), and when in (\ref{w1}) $\Psi=0$, it is equivalent (or at least is very close) to the latter. Certainly \cite{HK} is not the only place where the assertions of Theorem \ref{w4} may be found.
\begin{remark}\label{x5}
The effective equation is not uniquely defined in the sense that there are other stochastic equations for a curve $a(\tau)\in\mathbb{C}^{n}$ such that their solutions satisfy assertions (i) and (iii) of Theorem {\rm{\ref{w4}}}. See Section {\rm{3}} below, where we provide another effective equation with these properties.
\end{remark}
If (A3$^\prime$) holds, then solutions of (\ref{w3}) are defined for all $\tau\geq0$. In this case we will also use another assumption.

\textbf{Assumption }
(A4): The effective equation (\ref{w3}) is mixing with some stationary measure, and for each $M>0$ and any $v^{1}, v^{2}\in \overline{B}_{M}(\mathbb{C}^{n})$
\begin{equation*}
\|\mathcal{D}(a(\tau;v^{1}))-\mathcal{D}(a(\tau;v^{2}))\|^{*}_{L,\mathbb{C}^{n}}\leq g_{M}(\tau),
\end{equation*}
where $g$ is a continuous function of $(M, \tau)$ that tends to zero as $\tau\rightarrow\infty$.

We recall that for any two measures $\mu_{1}$ and $\mu_{2}$ on $\mathbb{C}^{n}$ the dual-Lipschitz distance between them is defined as
\begin{equation*}
\|\mu_{1}-\mu_{2}\|_{L,\mathbb{C}^{n}}^{*}:=\sup\limits_{f\in Lip_{0}(\mathbb{C}^{n},\mathbb{R}), \ \mathcal{C}^{0}(f)\leq1}\big|\int f{\rm{d}}\mu_{1}-\int f{\rm{d}}\mu_{2}\big|
\end{equation*}
(see Notation). Concerning the uniform in time convergence in items (ii) and (iii), in Sections 7, 8 of \cite{HK} the following result is established.
\begin{theorem}\label{w5}
Under Assumptions {\rm{A1, A2, A3$^\prime$, A4}}, for any $v_{0}\in\mathbb{C}^{n}$
\begin{enumerate}
\item $\lim\limits_{\varepsilon\rightarrow0}\sup\limits_{\tau\geq0}\|\mathcal{D}(a^{\varepsilon}(\tau;v_{0}))-\mathcal{D}(a(\tau;v_{0}))\|^{*}_{L,\mathbb{C}^{n}}=0$;
\item
    $\lim\limits_{\varepsilon\rightarrow0}\sup\limits_{\tau\geq0}\|\mathcal{D}(I(v^{\varepsilon}(\tau;v_{0})))-\mathcal{D}(I(a(\tau;v_{0})))\|^{*}_{L,\mathbb{C}^{n}}=0$.
\end{enumerate}
\end{theorem}

The assumptions, imposed in the theorems above are not too restrictive. In particular, in \cite{HK}, Proposition 9.4, it is shown that if the dispersion matrix $\Psi$ in (\ref{w1}) is a non-singular constant matrix, the drift satisfies $P\in Lip_{m_{0}}(\mathbb{C}^{n},\mathbb{C}^{n})$ for some $m_{0}\in\mathbb{N}$, and $P(v)\cdot v\leq-\alpha_{1}|v|+\alpha_{2}$ for some constants $\alpha_{1}>0$ and $\alpha_{2}\in\mathbb{R}$, then the assumptions of Theorem \ref{w4} and Theorem \ref{w5} hold.

\subsection{Averaged equation for actions}
By applying It\^{o}'s formula to $I_{k}(v^{\varepsilon}(\tau))$, where $v^{\varepsilon}(\tau)$ solves (\ref{w1}), we obtain
\begin{equation}\label{x8}
{\rm{d}} I_{k}(v^{\varepsilon}(\tau))=v^{\varepsilon}_{k}\cdot P_{k}(v^{\varepsilon}){\rm{d}} \tau+v^{\varepsilon}_{k}\cdot \left(\sum\limits_{l}\Psi_{kl}(v^{\varepsilon}){\rm{d}}\beta_{l}(\tau)\right)+\sum\limits_{l}\left|\Psi_{kl}(v^{\varepsilon})\right|^{2}{\rm{d}}\tau, \quad k=1,\ldots,n,
\end{equation}
and the angles $\varphi_{k}(v^{\varepsilon}(\tau))={\rm{Arg}}(v_{k}^{\varepsilon}(\tau))$ satisfy equations
\begin{equation}\label{x4}
{\rm{d}} \varphi_{k}(v^{\varepsilon}(\tau))=-\varepsilon^{-1}\lambda_{k}{\rm{d}} \tau+ O(1)(\varepsilon\rightarrow0), \quad k=1,\ldots,n.
\end{equation}
So in the action-angle variables $(I,\varphi)$ equations (\ref{w1}) become a fast-slow system. Formally applying the stochastic averaging to equations (\ref{x8}), (\ref{x4}) (see Sections 2, 3 in \cite{HKP}\footnote{Paper \cite{HKP} is written using real coordinates, so its results have to be adjusted to the complex setting. That work deals with perturbations of nonlinear systems, but its results apply to linear systems (\ref{w1}) with non-resonant frequency vectors $\Lambda$.}\ and references therein) we get the averaged equation for the vector of actions
\begin{equation}\label{z2}
{\rm{d}} I(\tau)=F(I){\rm{d}}\tau+K(I){\rm{d}}W(\tau), \quad I(\tau)\in\mathbb{R}_{+}^{n}.
\end{equation}
Here $W(\tau)$ is the standard Wiener process in $\mathbb{R}^{n}$, $F(I)$ is the averaging of the real-valued vector field with components $v_{k}\cdot P_{k}(v)+\sum_{l}\left|\Psi_{kl}(v)\right|^{2}$ in angles $\varphi=(\varphi_{1},\ldots,\varphi_{n})$, and the matrix $K(I)$ is obtained by the rules of stochastic averaging as $K(I)=\sqrt{S(I)}$, where the diffusion matrix $S(I)$ is the averaging in angles $\varphi$ of the real matrix with elements
\begin{align}\label{e3}
\sum\limits_{l}(v_{k}\overline{\Psi_{kl}}(v))\cdot(v_{j}\overline{\Psi_{jl}}(v)).
\end{align}

Now, let us assume that in addition to Assumptions A1, A3 the diffusion matrix $\Psi(v)\Psi^{*}(v)$ in {\rm{(\ref{w1})}} satisfies the uniform ellipticity condition. That is, there exists $\lambda>0$ such that
\begin{align}\label{x7}
\lambda|\xi|^{2}\leq \Psi(v)\Psi^{*}(v)\xi\cdot\xi\leq\lambda^{-1}|\xi|^{2}, \qquad \forall v, \xi\in\mathbb{C}^{n}.
\end{align}
Then, as is shown in \cite{HKP}, Section 6, equation (\ref{z2}) partially describes the limiting dynamics of actions $I_{k}(v^{\varepsilon})$ as $\varepsilon\rightarrow0$ in the following sense:
\begin{prop}\label{z5}
Under the above assumptions, for any $v_{0}\in\mathbb{C}^{n}$ the collection of laws of processes $\{I(v^{\varepsilon}(\tau;v_{0})),\tau\in[0,T]\}$, $0<\varepsilon\leq1$, is tight in $\mathcal{P}(C([0,T],\mathbb{R}_{+}^{n}))$. For any sequence $\varepsilon_{j}\rightarrow0$ such that $\mathcal{D}(I(v^{\varepsilon_{j}}(\cdot;v_{0})))\rightharpoonup Q^{0}\in\mathcal{P}(C([0,T],\mathbb{R}_{+}^{n}))$, the limit $Q^{0}$ is the law of a weak solution $I(\tau)$, $\tau\in[0,T]$, of the averaged equation {\rm{(\ref{z2})}}, equal $I(v_{0})$ at $\tau=0$.
\end{prop}

Naturally if equation (\ref{z2}) with the initial condition $I(0)=I(v_{0})$ has a unique solution, then the measure $Q^{0}$ is its law, and Proposition \ref{z5} implies
\begin{coro}\label{x9}
Under the assumptions of Proposition {\rm{\ref{z5}}}, if equation {\rm{(\ref{z2})}} has a unique solution $I(\tau)$ such that $I(0)=I(v_{0})$, then in Proposition {\rm{\ref{z5}}} the measure $Q^{0}$ is its law, and the convergence there holds as $\varepsilon\rightarrow0$.
\end{coro}
We note that results of work \cite{FW} (obtained by an argument, different from that in \cite{HKP}) imply assertions of Proposition \ref{z5} and Corollary \ref{x9} for solutions of (\ref{w1}) till they stay in a domain $\{v: |v_{k}|\geq\delta\}$, for any fixed $\delta>0$.

A-priori the coefficients of equation (\ref{z2}) are not locally Lipschitz functions of $I$. But as is shown in \cite{HK}, Section 6, if the coefficients of equation (\ref{w1}) are $C^{2}$-smooth, then in view of a result of Whitney \cite{W} the drift $F(I)$ in equation (\ref{z2}) is a $C^{1}$-function of $I$. Same argument shows that in this case the diffusion matrix $S(I)$ also is $C^{1}$-smooth. It degenerates when $I_{k}$ vanishes, so $K(I)=\sqrt{S(I)}$ is only a H\"{o}lder-$\frac{1}{2}$ continuous matrix function. The uniqueness for equations (\ref{z2}) with such dispersion matrices $K$ is a delicate question. If in equation (\ref{w1}) $\Psi$ is a constant matrix, then the elements of matrix $S(v)$ (see (\ref{e3})) are given by the averaging over $\varphi\in\mathbb{T}^{n}$ of functions $\sum_{l}\left(\sqrt{2I_{k}}{\rm{e}}^{i\varphi_{k}}\overline{\Psi_{kl}}\right)\cdot
\left(\sqrt{2I_{j}}{\rm{e}}^{i\varphi_{j}}\overline{\Psi_{jl}}\right)$, which equal
\begin{equation*}
\delta_{k-j}2I_{k}b_{k}^{2}, \quad b_{k}^{2}=\sum_{l}|\Psi_{kl}|^{2}\geq0,
\end{equation*}
(cf. Example \ref{e4}). So $K(I)={\rm{diag}}\{b_{k}\sqrt{2I_{k}}\}$. In this case, according to Theorem 1 in \cite{YW}, equation (\ref{z2}) with a prescribed $I(0)\in\mathbb{R}_{+}^{n}$ has a unique solution and Corollary \ref{x9} applies. When the matrix $\Psi$ is not constant, $K(I)=\sqrt{S(I)}$ is a matrix function with complicated singularities at $\partial \mathbb{R}_{+}^{n}$, and we are not aware of any result which would imply the uniqueness for equation (\ref{z2}). On the contrary, S. Watanabe and T. Yamada in \cite{WY} provided examples of equations (\ref{z2}) with some matrices $K(I)$ which degenerate at $\partial\mathbb{R}_{+}^{n}$ and are H\"{o}lder\text{-}$\frac{1}{2}$ continuous, for which a solution of (\ref{z2}) with a prescribed $I(0)$ is not unique.

Thus available results allow to use the averaged equation (\ref{z2}) to describe the limiting dynamics of actions $I_{k}(v^{\varepsilon}(\tau))$ if (apart from A1 and A3) the drift $P(v)$ in equation (\ref{w1}) is $C^{2}$-smooth and the dispersion $\Psi$ is a constant non-degenerate matrix. At the same time, due to item (ii) of Theorem \ref{w4}, the effective equation (\ref{w3}) describes the limiting dynamics if, apart from A1 and A3, the mild restriction A2 holds. Moreover, direct calculation in \cite{HKP}, Proposition 7.3 shows that:
\begin{prop}\label{z8}
The action vector $I(a(\tau))$ of a solution for {\rm{(\ref{w3})}} is a weak solution for equation {\rm{(\ref{z2})}}.
\end{prop}

\section{Averaging theorem for the modified effective equation}
Recall that for a complex function $f(z)$ of a complex variable $z=x+iy$, its derivatives $\frac{\partial f}{\partial z}$ and $\frac{\partial f}{\partial \bar{z}}$ are defined as $\frac{\partial f}{\partial z}=\frac{1}{2}(\frac{\partial f}{\partial x}-i\frac{\partial f}{\partial y})$ and $\frac{\partial f}{\partial \bar{z}}=\frac{1}{2}(\frac{\partial f}{\partial x}+i\frac{\partial f}{\partial y})$. A vector field $\widetilde{P}(v)$ in $\mathbb{C}^{n}$ is called \textit{hamiltonian} if
\[
\widetilde{P}_{k}(v)=i\frac{\partial}{\partial \bar{v}_{k}}h(v), \quad k=1,\ldots,n,
\]
where $h(v)$ is a $C^{1}$\text{-}smooth real function, called the Hamiltonian of $\widetilde{P}$.
In this case the $k$-th component of the averaged field is
\begin{align*}
\langle\langle \widetilde{P}\rangle\rangle_{k}(v)=\frac{1}{(2\pi)^{n}}\int_{\mathbb{T}^{n}}{\rm{e}}^{i\omega_{k}}i(\frac{\partial}{\partial \bar{v}_{k}}h)(\Phi_{-\omega}v){\rm{d}}\omega=\frac{i}{(2\pi)^{n}}\int_{\mathbb{T}^{n}}\frac{\partial}{\partial \bar{v}_{k}}h(\Phi_{-\omega}v){\rm{d}}\omega=i\frac{\partial}{\partial \bar{v}_{k}}\langle h\rangle(v)
\end{align*}
(we recall (\ref{z7}) and (\ref{e2})). So $\langle\langle \widetilde{P}\rangle\rangle$ is a hamiltonian field with the averaged Hamiltonian $\langle h\rangle$. For any $k$ consider the $k$-th complex coordinate $v_{k}=x_{k}+iy_{k}=r_{k}{\rm{e}}^{i\varphi_{k}}$ and replace it with two real coordinates $I_{k}=\frac{1}{2}r_{k}^{2}$ and $\varphi_{k}$. Since $\langle h\rangle$ does not depend on $\varphi_{k}$, then we have
\begin{align*}
\frac{\partial}{\partial \bar{v}_{k}}\langle h\rangle(v)=\frac{1}{2} \left(x_{k}\frac{\partial\langle h\rangle}{\partial I_{k}}+iy_{k}\frac{\partial\langle h\rangle}{\partial I_{k}}\right)=\frac{1}{2}\frac{\partial\langle h\rangle}{\partial I_{k}}\left(x_{k}+iy_{k}\right).
\end{align*}
So $(i\frac{\partial}{\partial \bar{v}_{k}}\langle h\rangle(v))\cdot v_{k}\equiv 0$ for each $k$. Thus, if in equation (\ref{w1})
\begin{equation*}
P=P^{1}+P^{2},
\end{equation*}
where the vector filed $P^{2}$ is hamiltonian, then $P^{2}$ gives no contribution to equation (\ref{x8}), as well as to the averaged equation for actions (\ref{z2}). So if Corollary \ref{x9} applies, then the hamiltonian component of the drift term $P(v)$ does not affect the long time dynamics of actions $I(v^{\varepsilon})$. But Corollary \ref{x9} is proved only for a small class of equations (\ref{w1}). Our goal in this section is to show that the same conclusion concerning the hamiltonian component of $P(v)$ holds if the dispersion matrix $\Psi(v)$ is uniformly elliptic and we use Theorem \ref{w4}.(iii) and Theorem \ref{w5}.(ii) to describe the limiting dynamics of the actions via the effective equation (\ref{w3}). For this end, apart from equation (\ref{w3}), we consider a modified effective equation
\begin{equation}\label{w6}
{\rm{d}} a_{k}(\tau)=\langle\langle P^{1}\rangle\rangle_{k}(a(\tau)){\rm{d}} \tau+\sum\limits_{l}B_{kl}(a(\tau)){\rm{d}}\beta_{l}(\tau), \quad k=1,\ldots,n,
\end{equation}
where $\langle\langle P^{1}\rangle\rangle$ is the averaging of the non-hamiltonian part $P^{1}$ of the vector filed $P$. This is an effective equation for equation (\ref{w1}), where $P(v)$ is replaced by $P^{1}(v)$. So if Assumptions A1-A3 hold for the latter equation, then by Theorem \ref{w4}.(i) equation (\ref{w6}) is well-posed. In the following, we use the left superscript ${}^1\cdot$ to indicate that we are considering the modified effective equation. For example, we denote by $^{1}a(\tau)$ and $^{1}I(\tau)=I(^{1}a(\tau))$ a solution of equation (\ref{w6}) and its action-vector. Then we have

\begin{theorem}\label{z6}
If the dispersion matrix $\Psi$ in equation {\rm{(\ref{w1})}} satisfy {\rm{(\ref{x7})}}, and Assumptions {\rm{A1, A3}} hold for equation {\rm{(\ref{w1})}} as well as for that equation with $P$ replaced by $P^{1}$, then assertions (i) and (iii) of Theorem {\rm{\ref{w4}}} hold for solutions of the modified effective equation {\rm{(\ref{w6})}}. In addition, if Assumptions {\rm{A3$^\prime$, A4}} are valid for both equations {\rm{(\ref{w1})}} and {\rm{(\ref{w1})}}$_{P:=P^{1}}$, then assertion (ii) of Theorem {\rm{\ref{w5}}} stays true for equation {\rm{(\ref{w6})}} as well.
\end{theorem}

Note that if Corollary \ref{x9} applies, then the assertion of the theorem holds trivially since by Proposition \ref{z8} actions $^{1}I_{k}(\tau)$ and $I_{k}(a(\tau))$ both satisfy equation (\ref{z2}) which has a unique solution. To prove the result in general case, we first for any $\delta>0$ construct a process $\widetilde{a}^{\delta}(\tau)\in\mathbb{C}^{n}$ such that almost surely $I(\widetilde{a}^{\delta}(\tau))=I(a(\tau))$ for all $\tau$ and $\widetilde{a}^{\delta}(\tau)$ satisfies (\ref{w6}) for $\tau$ outside a finite system of segments whose total length becomes small with $\delta$, and then show that as $\delta\rightarrow0$, $\widetilde{a}^{\delta}$ converges in distribution to a weak solution of (\ref{w6}). Since the latter is unique, the assertion will follow.
Now we present a complete proof.

\begin{proof}
The argument below uses some constructions from \cite{HKP,KP}.

 \textbf{Step 1:} \emph{Modifying the equations for large amplitudes.}

For any $R\in\mathbb{N}$, define the stopping time
\[
\tau_{R}=\inf\{\tau\in[0,T]\big||a(\tau)|^{2}\geq R\}.
\]
(If the set on the right-hand side is empty, then we take $\tau_{R}=T$). Then we consider the cut-off equation (\ref{w3})$_{R}$, which is equation (\ref{w3}) for $\tau\leq\tau_{R}$ and is the trivial system
\begin{equation}\label{e6}
{\rm{d}} a_{k}(\tau)={\rm{d}}\beta_{k}(\tau), \quad k=1,\ldots,n
\end{equation}
for $\tau\geq\tau_{R}$. Denote a solution of (\ref{w3})$_{R}$ by $a_{R}=(a_{R1},a_{R2},\ldots,a_{Rn})$. Its action $I_{R}(\tau)=(I_{Rk})_{1\leq k\leq n}:=(\frac{1}{2}|a_{Rk}|^{2})_{1\leq k\leq n}$ solves (\ref{z2})$_{R}$, which is equal to equation (\ref{z2}) for $\tau\leq\tau_{R}$ and to equations for actions of the trivial system above
\begin{equation}\label{x1}
{\rm{d}} I_{k}(\tau)={\rm{d}}\tau+\sqrt{2I_{k}} {\rm{d}}W_{k}, \quad k=1,\ldots,n
\end{equation}
for $\tau\geq\tau_{R}$. Here $\{W_{k}\}$ are independent standard real Wiener precesses. We define the cut-off equation (\ref{w6})$_{R}$ similarly, and denote its solution and the action of the latter by $^{1}a_{R}(\tau)$ and $^{1}I_{R}(\tau)$.

\textbf{Step 2:} \emph{Construction of a process $\widetilde{a}_{R}$.}

Our goal is to construct a new process $\widetilde{a}_{R}(\tau)=(\widetilde{a}_{R1},\ldots,\widetilde{a}_{Rn})$ such that
\begin{enumerate}[{\rm \qquad 1)}]
\item $\widetilde{a}_{R}(\tau)$ solves (\ref{w6})$_{R}$;
\item $\mathcal{D}(\widetilde{I}_{R})=\mathcal{D}(I_{R})$, where $\widetilde{I}_{R}=(\widetilde{I}_{Rk})_{1\leq k\leq n}=(\frac{1}{2}|\widetilde{a}_{Rk}|^{2})_{1\leq k\leq n}$.
\end{enumerate}

Let us fix $\delta>0$. For any curve $I(\tau)\in\mathbb{R}_{+}^{n}$, we denote $[I(\tau)]=\min_{1\leq k\leq n}\{I_{k}(\tau)\}$. We first construct a process $\widetilde{a}_{R}^{\delta}(\tau)=(\widetilde{a}_{R1}^{\delta},\ldots,\widetilde{a}_{Rn}^{\delta})$, such that $\widetilde{I}^{\delta}_{Rk}(\tau):=\frac{1}{2}|\widetilde{a}^{\delta}_{Rk}(\tau)|^{2}\equiv I_{Rk}(\tau)$ for all $k$ a.s. Then we will prove that its limit as $\delta\rightarrow0$ is the process $\widetilde{a}_{R}(\tau)$ that we need.

For definiteness, we assume $[I(v_{0})]>\delta$ and set $\tau_{0}^{+}=0$.
\begin{enumerate}
\item We take for $\widetilde{a}_{R}^{\delta}(\tau)$ a solution of (\ref{w6})$_{R}$ with the initial data $v_{0}$ until $\tau_{1}^{-}$, where $\tau_{1}^{-}$ is the first moment after $\tau_{0}^{+}$ such that $[\widetilde{I}^{\delta}_{R}(\tau)]\leq\delta$ (if this never happens, we take $\tau_{1}^{-}=T$). See Figure \ref{pic1}.
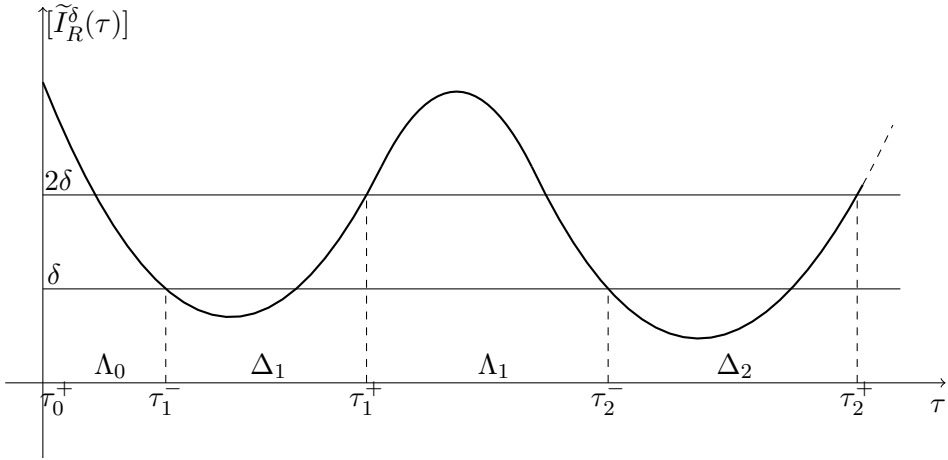
\begin{figure}[h]
{$\!\!\!\!\!\!\!\!\!\!\!\!\!\!\!\!\!$
\begin{tikzpicture}
\draw[, ->, black] (0,1) -- (12.5,1);
\draw[, ->, black] (0.5,0) -- (0.5,6);
\draw[thin] (0.5,2.25) -- (11.9,2.25);
\draw[thin] (0.5,3.5) -- (11.9,3.5);
\draw[thin, dashed] (2.134,1) -- (2.134,2.25);
\draw[thin, dashed] (4.803,1) -- (4.803,3.5);
\draw[thin, dashed] (8.015,1) -- (8.015,2.25);
\draw[thin, dashed] (11.326,1) -- (11.326,3.5);
\node at (12.4,0.7) {\color{black} $\tau$};
\node at (1.1,5.8) {\color{black} $[\widetilde{I}^{\delta}_{R}(\tau)]$
};
\node at (2.134,0.8) {\color{black} $\tau_1^-$};
\node at (4.803,0.8) {\color{black} $\tau_1^+$};
\node at (0.65,2.45) {\color{black} $\delta$};
\node at (0.71,3.7) {\color{black} $2\delta$};
\node at (0.69,0.8) {\color{black} $\tau_0^+$};
\node at (1.4,1.22) {\color{black} $\Lambda_0$};
\node at (3.5,1.22) {\color{black} $\Delta_1$};
\node at (6.5,1.22) {\color{black} $\Lambda_1$};
\node at (9.7,1.22) {\color{black} $\Delta_2$};
\node at (8.015,0.8) {\color{black} $\tau_2^-$};
\node at (11.326,0.8) {\color{black} $\tau_2^+$};
\draw[thick, black, domain=0.5:5] plot (\x, {0.5*(\x-0.5)*(\x-0.5)-2.5*(\x-0.5) +5});
\draw[thick,  black, domain=5:7] plot (\x, {-(\x-5)*(\x-5)+2*(\x-5) +3.875});
\draw[thick,  black, domain=7:11.4] plot (\x, {-0.008*(\x-7)*(\x-7)*(\x-7)+0.5*(\x-10)*(\x-10)+ 0.9*(\x-7) -0.625});
\draw[dashed,  black, domain=11.4:11.8] plot (\x, {-0.008*(\x-7)*(\x-7)*(\x-7)+0.5*(\x-10)*(\x-10)+ 0.9*(\x-7) -0.625});
\end{tikzpicture}
}
  \caption{A typical behaviour of the stopping times  $\tau_j^\pm$}\label{pic1}
\end{figure}

\item We know $\widetilde{a}_{R}^{\delta}(\tau)$ at point $\tau_{1}^{-}$. Below in Lemma \ref{w9} we show that $\widetilde{I}^{\delta}_{R}(\tau)=I_{R}(\tau)$ for all $\tau\in[\tau_{0}^{+},\tau_{1}^{-}]$, so $\widetilde{I}^{\delta}_{R}(\tau_{1}^{-})=I_{R}(\tau_{1}^{-})$. Thus, $\widetilde{a}^{\delta}_{R}(\tau_{1}^{-})=\Phi_{\theta(\tau_{1}^{-})}a_{R}(\tau_{1}^{-})$ for a suitable angle-vector $\theta(\tau_{1}^{-})\in \mathbb{T}^{n}$. Then we set $\widetilde{a}^{\delta}_{R}(\tau)=\Phi_{\theta(\tau_{1}^{-})}a_{R}(\tau)$ for $\tau\in[\tau_{1}^{-},\tau_{1}^{+}]$, where $\tau_{1}^{+}$ is the first moment after $\tau_{1}^{-}$ such that $[\widetilde{I}^{\delta}_{R}(\tau)]\geq2\delta$ (if this never happens, we take $\tau_{1}^{+}=T$).

\item Now we come to $\tau_{1}^{+}$. We take for $\widetilde{a}_{R}^{\delta}(\tau)$ a solution of (\ref{w6})$_{R}$ with the initial data $\widetilde{a}_{R}^{\delta}(\tau_{1}^{+})$ until $\tau_{2}^{-}$, where $\tau_{2}^{-}$ is the first moment after $\tau_{1}^{+}$ such that $[\widetilde{I}^{\delta}_{R}(\tau)]\leq\delta$ (if this never happens, we take $\tau_{2}^{-}=T$), etc.
\end{enumerate}

We repeat the above steps until $\tau=T$. It is easy to see from the equation for $\widetilde{a}_{R}^{\delta}(\tau)$ that a.s., either $\tau_{j}^{-}=T$ or $\tau_{j}^{+}=T$ for some finite $j$. Thus we obtain a process $\widetilde{a}_{R}^{\delta}(\tau)$, $0\leq \tau\leq T$, such that $\widetilde{I}_{R}^{\delta}(\tau)=I_{R}(\tau)$ a.s. In the following, we denote $\Delta_{j}=[\tau_{j}^{-},\tau_{j}^{+}]$ and $\Lambda_{j}=[\tau_{j}^{+},\tau_{j+1}^{-}]$. Then $[\widetilde{I}^{\delta}_{R}(\tau)]\geq\delta$ on $\Lambda_{j}$, $[\widetilde{I}^{\delta}_{R}(\tau)]\leq2\delta$ on $\Delta_{j}$ for each $j\geq0$ and
\begin{equation}\label{x2}
\Lambda_{0}\leq \Delta_{1}\leq\Lambda_{1}\leq \Delta_{2}\leq\ldots.
\end{equation}
The process $a_{R}^{\delta}(\tau)$ is defined on the union of the intervals (\ref{x2}) which equals $[0,T]$, a.s.

\begin{lemma}\label{w9}
For $\tau\in\Lambda_{0}\cup\Delta_{1}$, we have $\widetilde{I}^{\delta}_{R}(\tau)=I_{R}(\tau)$ a.s.
\end{lemma}
\begin{proof}
For $0\leq\tau\leq\tau_{1}^{-}$ let $\widetilde{\tau}$ be the stopping time $\widetilde{\tau}=\tau_{1}^{-}\wedge\min\{\tau: [I_{R}(\tau)]=\delta\}$. Then $\widetilde{\tau}>0$. For $0\leq\tau\leq\widetilde{\tau}$ the curves $\widetilde{I}_{R}^{\delta}(\tau)$ and $I_{R}(\tau)$ stay in domain $Q=\{I: I_{j}\geq\delta \ \forall j\}$ and satisfy there equation (\ref{z2})$_{R}$. Since in $Q$ that equation has locally Lipschitz coefficients, then $\widetilde{I}_{R}^{\delta}(\tau)=I_{R}(\tau)$ for $\tau\leq\widetilde{\tau}$. Thus $\widetilde{\tau}=\tau_{1}^{-}$, and for $\tau\in\Lambda_{0}$ the assertion follows.

For $\tau\in\Delta_{1}$, $\widetilde{I}^{\delta}_{R}(\tau)=I_{R}(\tau)$ since $\widetilde{a}^{\delta}_{R}(\tau)=\Phi_{\theta(\tau_{1}^{-})}a_{R}(\tau)$.
\end{proof}

Using the construction (i)-(iii) and applying Lemma \ref{w9} to intervals $\Lambda_{j}$, $\Delta_{j}$ with $j\geq0$, we obtained a process $\widetilde{a}^{\delta}_{R}(\tau)$, $\tau\in[0,T]$, such that:

(1). $\widetilde{a}^{\delta}_{R}(\tau)$ solves (\ref{w6})$_{R}$ on $\bigcup_{j\geq0}\Lambda_{j}$;

(2). $\widetilde{I}^{\delta}_{R}(\tau)=I_{R}(\tau)$ on $[0,T]$ a.s.\\
Additionally, we have
\begin{lemma}\label{w7}
For every $k$, $\mathbf{E}\int_{0}^{\tau_{R}}1_{\{{I}_{Rk}(\tau)\leq\delta\}}(\tau){\rm{d}}\tau\rightarrow0$ as $\delta\rightarrow0$.
\end{lemma}
\begin{proof}
The solution $a_{R}(\tau)$ satisfies
\begin{equation*}
{\rm{d}} a_{Rk}(\tau)=1_{\tau\leq\tau_{R}}\left(\langle\langle P\rangle\rangle_{k}(a_{R}){\rm{d}} \tau+\sum\limits_{l}B_{kl}(a_{R}){\rm{d}}\beta_{l}(\tau)\right)+1_{\tau\geq\tau_{R}}{\rm{d}}\beta_{k}(\tau), \quad k=1,\ldots,n.
\end{equation*}
Since $\left|\langle\langle P\rangle\rangle(a_{R})\right|\leq C^{m_{0}}(P)(1+R)^{m_{0}}$ and $C E\leq B(a_{R})B^{*}(a_{R})\leq C^{-1}E$ as the matrix $\Psi$ satisfy (\ref{x7}), Theorem 2.2.4 in \cite{NK} implies that $\mathbf{E}\int_{0}^{\tau_{R}}1_{\{{I}_{Rk}(\tau)\leq\delta\}}(\tau){\rm{d}}\tau\leq C\delta^{\frac{1}{n}}$, where $C=C(R,n)$. Thus, the lemma follows.
\end{proof}

\textbf{Step 3:} \emph{The limit as $\delta\rightarrow0$.}

\begin{lemma}\label{x3}
For any fixed sequence $\delta_{j}\rightarrow0$ the family of measures $\{\mathcal{D}(\widetilde{a}^{\delta_{j}}_{R}(\tau)), j\geq1\}$ is tight in $\mathcal{P}(C([0,T];\mathbb{C}^{n}))$.
\end{lemma}
\begin{proof}
The equation for $\widetilde{a}^{\delta}_{R}(\tau)$ is
\begin{equation}\label{e5}
\begin{split}
&{\rm{d}} \widetilde{a}^{\delta}_{Rk}(\tau)=\sum\limits_{j}1_{\tau\leq\tau_{R}}1_{\tau\in\Lambda_{j}}\left(\langle\langle P^{1}\rangle\rangle_{k}(\widetilde{a}^{\delta}_{R}){\rm{d}}\tau+\sum\limits_{l}B_{kl}(\widetilde{a}^{\delta}_{R}){\rm{d}}\beta_{l}(\tau)\right)\\
&+\sum\limits_{j}1_{\tau\leq\tau_{R}}1_{\tau\in\Delta_{j}}\Phi_{\theta(\tau_{j}^{-})}\left(\langle\langle P\rangle\rangle_{k}(a_{R}){\rm{d}}\tau+\sum\limits_{l}B_{kl}(a_{R}){\rm{d}}\beta_{l}(\tau)\right)
+1_{\tau\geq\tau_{R}}{\rm{d}}\beta_{k}(\tau), \ \ k=1,\ldots, n,
\end{split}
\end{equation}
where $\widetilde{a}^{\delta}_{R}(\tau)=\Phi_{\theta(\tau_{j}^{-})}a_{R}(\tau)$ for $\tau\in\Delta_{j}$. Since $I(\widetilde{a}_{R}^{\delta}(\tau))=I(a_{R}(\tau))$, where for $\tau\leq\tau_{R}$ the norm of $a_{R}(\tau)$ is bounded by $\sqrt{R}$ and for $\tau\geq\tau_{R}$, $a_{R}(\tau)$ satisfies the trivial equation (\ref{e6}), then $\mathbf{E}\sup_{0\leq\tau\leq T}|\widetilde{a}^{\delta}_{R}(\tau)|\leq C(R)$. Next in view of (\ref{e5}) and Assumption A1 we have for any $0\leq\tau_{1}\leq\tau_{2}\leq T$ that
\begin{align*}
&\mathbf{E}\big|\widetilde{a}^{\delta}_{R}(\tau_{2})-\widetilde{a}^{\delta}_{R}(\tau_{1})\big|^{4}
\leq C\mathbf{E}\bigg|\sum\limits_{j}\int_{\tau_{1}}^{\tau_{2}}1_{\tau\leq\tau_{R}}\left(1_{\tau\in\Lambda_{j}}\langle\langle
P^{1}\rangle\rangle+1_{\tau\in\Delta_{j}}\Phi_{\theta(\tau_{j}^{-})}\langle\langle P\rangle\rangle\right){\rm{d}}\tau\bigg|^{4}\\
&\quad +C\mathbf{E}\bigg|\sum\limits_{j}\int_{\tau_{1}}^{\tau_{2}}
1_{\tau\leq\tau_{R}}\left(1_{\tau\in\Lambda_{j}}B+1_{\tau\in\Delta_{j}}\Phi_{\theta(\tau_{j}^{-})}B\right){\rm{d}}\beta(\tau)\bigg|^{4}
+C\mathbf{E}\bigg|\int_{\tau_{1}}^{\tau_{2}}1_{\tau\geq\tau_{R}}{\rm{d}}\beta(\tau)\bigg|^{4}\\
&\leq C(R)\left(|\tau_{2}-\tau_{1}|^{4}+|\tau_{2}-\tau_{1}|^{2}\right).
\end{align*}
Then by the Kolmogorov theorem on the H\"{o}lder continuity of a random process\footnote{Strictly speaking this fact is a consequence not of the Kolmogorov theorem, but of its proof. See Theorem 7 in Section 1.4 of \cite{NK02}.} and Prokhorov's theorem the assertion follows.
\end{proof}
So there exists a sequence $\delta_{l}\rightarrow0$ and a measure $Q_{R}^{0}$ in $C([0,T];\mathbb{C}^{n})$ such that $Q_{R}^{\delta_{l}}:=\mathcal{D}(\widetilde{a}^{\delta_{l}}_{R}(\tau))\rightharpoonup Q_{R}^{0}$ as $\delta_{l}\rightarrow0$. Then

\begin{lemma}\label{z4}
We have $Q_{R}^{0}=\mathcal{D}(\widetilde{a}_{R}(\tau))$, where $\widetilde{a}_{R}(\tau)$ is a unique weak solution of the cut-off modified effective equation {\rm{(\ref{w6})$_{R}$}} with initial data $v_{0}$, and $\mathcal{D}(I(\widetilde{a}_{R}))=\mathcal{D}(I(a_{R}))$. Here $a_{R}(\tau)$ is a solution of the cut-off effective equation {\rm{(\ref{w3})$_{R}$}} with initial data $v_{0}$.
\end{lemma}
\begin{proof}
We consider the natural filtered measurable space $(\widetilde{\Omega},\mathcal{B},\widetilde{\mathcal{F}}_{t})$, where $\widetilde{\Omega}=C([0,T];\mathbb{C}^{n})$, $\mathcal{B}$ is the Borel $\sigma$\text{-}algebra on $\widetilde{\Omega}$ and $\widetilde{\mathcal{F}}_{t}$ is its natural filtration. We set $\Delta=\bigcup_{j}\Delta_{j}$ and $\Lambda=\bigcup_{j}\Lambda_{j}$.
Denote
\begin{align*}
&N_{k}(a;\tau)=a_{k}(\tau)-\int_{0}^{\tau}1_{s\leq\tau_{R}}\langle\langle P^{1}\rangle\rangle_{k}(a(s)){\rm{d}}s, \qquad a\in\widetilde{\Omega};\\
&N_{k}^{\delta}(\tau)\!=\!\widetilde{a}^{\delta}_{Rk}(\tau)\!-\!\int_{0}^{\tau}1_{s\leq\tau_{R}}1_{s\in\Lambda}\langle\langle P^{1}\rangle\rangle_{k}(\widetilde{a}^{\delta}_{R}(s)){\rm{d}}s\!-\!\sum\limits_{j}\int_{0}^{\tau}1_{s\leq\tau_{R}}1_{s\in\Delta_{j}}\Phi_{\theta(\tau_{j}^{-})}
\langle\langle P\rangle\rangle_{k}(a_{R}(s)){\rm{d}}s;\\
&M_{k}(\tau)=\sum\limits_{j}\int_{0}^{\tau}\left[1_{s\leq\tau_{R}}1_{s\in\Delta_{j}}\langle\langle P^{1}\rangle\rangle_{k}(\widetilde{a}^{\delta}_{R}(s))-1_{s\leq\tau_{R}}1_{s\in\Delta_{j}}\Phi_{\theta(\tau_{j}^{-})}\langle\langle P\rangle\rangle_{k}(a_{R}(s))\right]{\rm{d}}s.
\end{align*}
Due to (\ref{e5}), the process $N_{k}^{\delta}(\tau)$ is a martingale. Next we estimate $M(\tau)$:
\begin{equation}\label{e1}
\begin{split}
&\mathbf{E}\sup\limits_{0\leq\tau\leq T}\big|M_{k}(\tau)\big|\\
\leq &\mathbf{E}\int_{0}^{\tau_{R}}\big|1_{s\in\Delta}\langle\langle P^{1}\rangle\rangle_{k}(\widetilde{a}^{\delta}_{R}(s))\big|{\rm{d}}s+\mathbf{E}\int_{0}^{\tau_{R}}\big|1_{s\in\Delta}\langle\langle P\rangle\rangle_{k}(a_{R}(s))\big|{\rm{d}}s\\
\leq &C(R)\left(\mathbf{E}\int_{0}^{\tau_{R}}1_{\{[\widetilde{I}^{\delta}_{R}(\tau)]\leq2\delta\}}(\tau){\rm{d}}\tau\right)^{\frac{1}{2}}
\end{split}
\end{equation}
since $\widetilde{a}^{\delta}_{R}(\tau)=\Phi_{\theta(\tau_{j}^{-})}a_{R}(\tau)$ for $\tau\in\Delta_{j}$. In view of Lemma \ref{w7}, the right-hand side goes to $0$ with $\delta$.

We claim that $N_{k}(a;\tau)$ is a $Q_{R}^{0}$\text{-}martingale on the space $(\widetilde{\Omega},\mathcal{B},\widetilde{\mathcal{F}}_{t})$. To prove that for any $0\leq\tau_{1}\leq\tau_{2}\leq T$ and a bounded continuous function $f$ on $\widetilde{\Omega}$ such that $f(a)$ depends only on $a(\tau)$ for $\tau\in[0,\tau_{1}]$, we have to show that
\begin{align*}
\mathbf{E}^{Q_{R}^{0}}\left(\left(N_{k}(a;\tau_{2})-N_{k}(a;\tau_{1})\right)f(a)\right)=0.
\end{align*}
Since $Q_{R}^{\delta_{l}}\rightharpoonup Q_{R}^{0}$, and $\mathbf{E}^{Q_{R}^{\delta}}(N_{k}(a;\tau)f(a))=\mathbf{E}N_{k}(\widetilde{a}_{R}^{\delta}(\cdot);\tau)f(\widetilde{a}_{R}^{\delta}(\cdot))$, then the left-hand side equals
\begin{align*}
&\lim_{\delta_{l}\rightarrow0}\mathbf{E}^{Q_{R}^{\delta_{l}}}\left(\left(N_{k}(a;\tau_{2})-N_{k}(a;\tau_{1})\right)f(a)\right)\\
=&\lim_{\delta_{l}\rightarrow0}\mathbf{E}\left(\left(\widetilde{a}^{\delta_{l}}_{Rk}(\tau_{2})-\widetilde{a}^{\delta_{l}}_{Rk}(\tau_{1})-\int_{\tau_{1}}^{\tau_{2}}
1_{s\leq\tau_{R}}\langle\langle P^{1}\rangle\rangle_{k}(\widetilde{a}^{\delta_{l}}_{R}(s)){\rm{d}}s\right)f(\widetilde{a}_{R}^{\delta_{l}}(\cdot))\right)\\
=&\lim_{\delta_{l}\rightarrow0}\mathbf{E}\left((M_{k}(\tau_{1})-M_{k}(\tau_{2}))f(\widetilde{a}_{R}^{\delta_{l}}(\cdot))\right)
\leq C\lim_{\delta_{l}\rightarrow0}\mathbf{E}\big|M_{k}(\tau_{1})-M_{k}(\tau_{2})\big|=0,
\end{align*}
where the second equality holds due to the fact that $N_{k}^{\delta_{l}}(\tau)$ is a martingale and the last equality follows from the estimate (\ref{e1}). Then $N_{k}(\tau)$ is a $Q_{R}^{0}$\text{-}martingale on $(\widetilde{\Omega},\mathcal{B},\widetilde{\mathcal{F}}_{t})$.

Due to (\ref{e5}), for any $1\leq k,l\leq n$ the process $N_{k}^{\delta_{l}}(\tau)\overline{N_{l}^{\delta_{l}}}(\tau)-2\sum_{j}\int_{0}^{\tau}1_{s\leq\tau_{R}}B_{kj}\overline{B_{lj}}{\rm{d}}s-2\int_{0}^{\tau}1_{s\geq\tau_{R}}1_{k=l}{\rm{d}}s$ is a martingale. By repeating the arguments above, we find that the process $N_{k}(\tau)\overline{N_{l}}(\tau)-2\sum_{j}\int_{0}^{\tau}1_{s\leq\tau_{R}}B_{kj}\overline{B_{lj}}{\rm{d}}s-2\int_{0}^{\tau}1_{s\geq\tau_{R}}1_{k=l}{\rm{d}}s$ is a $Q_{R}^{0}$\text{-}martingale on $(\widetilde{\Omega},\mathcal{B},\widetilde{\mathcal{F}}_{t})$.

Similarly, the fact that the process $N_{k}^{\delta_{l}}(\tau)N_{l}^{\delta_{l}}(\tau)$ is a martingale and the convergence $Q_{R}^{\delta_{l}}\rightharpoonup Q_{R}^{0}$ imply that $N_{k}(\tau)N_{l}(\tau)$ is a $Q_{R}^{0}$\text{-}martingale on $(\widetilde{\Omega},\mathcal{B},\widetilde{\mathcal{F}}_{t})$.

Therefore $\widetilde{a}_{R}(\tau)$ is a martingale solution of the cut-off modified effective equation (\ref{w6})$_{R}$ with initial data $v_{0}$ (see Definition 4.5 and Appendix B in \cite{HK} for martingale solutions in $\mathbb{C}^{n}$). Thus it is its weak solution. Since a solution of equation (\ref{w6})$_{R}$ with a given initial data is unique, then the solution $\widetilde{a}_{R}(\tau)$ is its unique weak solution. As $\widetilde{I}^{\delta}_{R}(\tau)=I(a_{R}(\tau))$ a.s., then $\mathcal{D}(I(\widetilde{a}_{R}))=\mathcal{D}(I(a_{R}))$.
\end{proof}
\textbf{Step 4:} \emph{The limit as $R\rightarrow\infty$.}

Similar to the proof of Lemma \ref{x3} and due to (\ref{w8}), the set of measures $\{\mathcal{D}(\widetilde{a}_{R}(\tau)),R\geq0\}$ is tight. Consider some limiting measure $Q^{0}$ as $R\rightarrow\infty$. By repeating the proof of Lemma \ref{z4}, we find that $Q^{0}=\mathcal{D}(\widetilde{a}(\tau))$, where $\widetilde{a}(\tau)$ is a weak solution of the modified effective equation (\ref{w6}). Since a solution of (\ref{w6}) with a given initial data is unique, then the solution $\widetilde{a}(\tau)$ is its unique weak solution. By the above it satisfies $\mathcal{D}(I(\widetilde{a}_{R}))=\mathcal{D}(I(a_{R}))$, for any $R\in\mathbb{N}$, then $\mathcal{D}(I(\widetilde{a}))=\mathcal{D}(I(a))$, where $a(\tau)$ is the solution of effective equation (\ref{w3}).
Therefore, if we replace the effective equation (\ref{w3}) with the modified effective equation (\ref{w6}), then the assertions of Theorem \ref{w4}.(iii) and Theorem \ref{w5}.(ii) still hold.
\end{proof}

\noindent{\bf{Acknowledgement}}
We thank Alexander Veretennikov for discussion of stochastic equations. The research of Jing Guo was supported by the China Scholarship Council (202306060105). The research of Zhenxin Liu was supported by National Key R\&D Program of China (No. 2023YFA1009200), NSFC (Grant 11925102), Liaoning Revitalization Talents Program (Grant XLYC2202042).


\begin{thebibliography}{xx}
\bibitem{BM}
N. N. Bogoliubov, Y. A. Mitropolsky, Asymptotic Methods in the Theory of Non-linear Oscillations. International Monographs on Advanced Mathematics and Physics, Hindustan Publishing Corp., Delhi, Gordon and Breach Science Publishers, Inc., New York, 1961, x+537 pp.

\bibitem{FW}
M. Freidlin, A. Wentzell, Averaging principle for stochastic perturbations of multifrequency systems. {\it Stoch. Dyn.} {\bf 3}  (2003),  no. 3, 393--408.

\bibitem{JKL2}
J. Guo, S. Kuksin, Averaging for stochastic nonlinear Schr\"{o}dinger equations. work in preparation.

\bibitem{HK}
G. Huang, S. Kuksin,  Averaging and mixing for stochastic perturbations of linear conservative systems. {\it
Uspekhi Mat. Nauk} {\bf 78}  (2023),  no. 4, 3--52. {\it
Russian Math. Surveys} {\bf 78}  (2023),  no. 4, 585--633.

\bibitem{HKP}
G. Huang, S. Kuksin, A. Piatnitski, Averaging for stochastic perturbations of integrable systems. arXiv:2307.07040.

\bibitem{JKW}
W. Jian, S. Kuksin, Y. Wu, Krylov-Bogolyubov averaging. {\it Uspekhi Mat. Nauk}  {\bf
75} (2020), no. 3, 37--54. {\it Russian Math. Surveys}  {\bf
75} (2020), no. 3, 427--444.

\bibitem{NK}
N. Krylov, Controlled Diffusion Processes. Springer-Verlag, New York-Berlin, 1980, xii+308 pp.

\bibitem{NK02}
N. Krylov, Introduction to the Theory of Random Processes. Grad. Stud. Math., 43, American Mathematical Society, Providence, RI, 2002, xii+230 pp.

\bibitem{KP}
S. Kuksin, A. Piatnitski, Khasminskii-Whitham averaging for randomly perturbed KdV equation. {\it J. Math. Pures Appl. (9)} {\bf 89} (2008),  no. 4, 400--428.

\bibitem{WY}
S. Watanabe, T. Yamada, On the uniqueness of solutions of stochastic differential equations. II. {\it J. Math. Kyoto Univ.} {\bf 11}  (1971), 553--563.

\bibitem{W}
H. Whitney, Differentiable even functions. II. {\it Duke Math. J.} {\bf 10}  (1943), 159--160.

\bibitem{YW}
T. Yamada, S. Watanabe, On the uniqueness of solutions of stochastic differential equations. {\it J. Math. Kyoto Univ.} {\bf 11}  (1971), 155--167.
\end{thebibliography}
\end{document}